\documentclass[11pt]{amsart}
\usepackage{setspace, amsmath, amsthm, amssymb, amsfonts, amscd, epic, graphicx, ulem, dsfont}
\usepackage[T1]{fontenc}
\usepackage{multirow}
\usepackage{bbm}
\usepackage{enumerate}

\makeatletter \@namedef{subjclassname@2010}{
  \textup{2010} Mathematics Subject Classification}
\makeatother

\newtheorem{thm}{Theorem}[section]
\newtheorem{cor}[thm]{Corollary}
\newtheorem{lem}[thm]{Lemma}
\newtheorem{pro}[thm]{Proposition}

\theoremstyle{remark}

\newtheorem{exa}[thm]{\textbf{Example}}

\theoremstyle{definition}

\newcommand{\R}{\mathbb{R}}

\newcommand{\C}{\mathbb{C}}

\begin{document}

\title[Commutativity of Unbounded Self-adjoint Operators]{Conditions Implying Commutativity of Unbounded Self-adjoint Operators and Related Topics}
\author[K. Gustafson and M. H. Mortad]{Karl Gustafson and MOHAMMED HICHEM MORTAD*}

\dedicatory{}
\thanks{* Corresponding author.}
\date{}
\keywords{Normal and Self-adjoint operators. Commutativity.  Fuglede
Theorem.}

\subjclass[2010]{Primary 47B25,
  47B15}

\address{Department of Mathematics, University of Colorado at Boulder, Campus
Box 395 Boulder, CO 80309-0395, USA.}
\email{Karl.Gustafson@colorado.edu, gustafs@euclid.colorado.edu.}

\address{Department of
Mathematics, University of Oran 1 (Ahmed Benbella), B.P. 1524, El
Menouar, Oran 31000, Algeria.\newline {\bf Mailing address}:
\newline Pr Mohammed Hichem Mortad \newline BP 7085 Seddikia Oran
\newline 31013 \newline Algeria}

\email{mhmortad@gmail.com, mortad@univ-oran.dz.}

\begin{abstract}
Let $B$ be a bounded self-adjoint operator and let $A$ be a
nonnegative self-adjoint unbounded operator. It is shown that if
$BA$ is normal, it must be self-adjoint and so must be $AB$.
Commutativity is necessary and sufficient for this result. If $AB$
is normal, it must be self-adjoint and $BA$ is essentially
self-adjoint. Although the two problems seem to be alike, two
different and quite interesting approaches are used to tackle them.
\end{abstract}

\maketitle

\section{Introduction}

In \cite{Mortad-PAMS2003}, the following result (among others) was
proved:

\begin{thm}\label{Mortad-2003 product self-adjoint}
Let $A$ be an unbounded self-adjoint operator and let $B$ be a
positive (or negative) bounded operator. If $AB$ (respectively $BA$)
is normal, then $AB$ (respectively $BA$) is self-adjoint.
\end{thm}

The foregoing results have applications for example to the problem
of commutativity up to a factor (see \cite{Chellali-Mortad}). They
also provide us with a tool for commutativity of self-adjoint
operators (see the proof of Theorem \ref{Mortad-2003 product
self-adjoint} in \cite{Mortad-PAMS2003}).

The "$AB$ case" was generalized in \cite{Mortad-PAMS2003} to the
case of two unbounded self-adjoint operators $A$ and $B$. Later in
\cite{Mortad-IEOT-2009} it was shown that under the same conditions,
the normality of $BA$ does not imply anymore its self-adjointness.
But, there are still two cases to look at, namely: Keeping $B$
bounded, but taking $A$ to be positive (both self-adjoint):
\begin{enumerate}
  \item Does $AB$ normal imply $AB$ self-adjoint?
  \item Does $BA$ normal imply $BA$ self-adjoint?
\end{enumerate}

These are the main questions asked in this paper. Commutativity
relations play a vital role in their consideration. Other issues
that arise are the domains and closedness of unbounded operators.

To prove the results we first assume basic notions and results on
unbounded operator theory. For basic references, see \cite{Con},
\cite{Kat}, \cite{RUD}, \cite{SCHMUDG-book-2012} and
\cite{Weidmann}. Recall that a densely defined unbounded operator
$A$ is: normal if it is closed and $AA^*=A^*A$, symmetric if
$A\subset A^*$ and self-adjoint if $A=A^*$. Observe that both
"symmetric" and "normal" are weaker than "self-adjoint", but one of
the easy and nice results is: A normal and symmetric operator must
be self-adjoint. Also, remember that the commutativity between a
bounded $B$ and an unbounded $A$ is expressed as $BA\subset AB$.

Before finishing the introduction, we recall some other results
(needed in the sequel) which cannot be considered as elementary.

The first is the well-known Fuglede-Putnam theorem:

\begin{thm}[for a proof, see \cite{Con}]\label{Fuglede Putnam classic}
If $T$ is a bounded operator and and $N$ and $M$ are unbounded and
normal, then
\[TN\subset MT \Longrightarrow TN^*\subset M^*T.\]
\end{thm}

\begin{cor}\label{Fuglede Putnam classic TN=MT} If $T$ is a bounded operator and and
$N$ and $M$ are unbounded and normal, then
\[TN=MT \Longrightarrow TN^*=M^*T.\]
\end{cor}

The next is a generalization of the Fuglede theorem.

\begin{thm}[Mortad, \cite{Mortad-PAMS2003}]\label{Fuglede-Mortad-2003}
Let $T$ be an unbounded self-adjoint operator with domain $D(T)$. If
$N$ is an unbounded normal operator such that $D(N)\subset D(T)$,
then
\[TN\subset N^*T\Longrightarrow TN^*\subset NT.\]
\end{thm}

We also note

\begin{lem}\label{(AB)*=B*A* B INVERTIBLE!!} If $A$ and $B$ are densely defined with
inverse $B^{-1}$ in $B(H)$. Then $(AB)^* = B^* A^*$.  In particular,
if $U$ is unitary, then
\[(UAU)^*=U^*(UA)^*=U^*A^*U^*.\]
\end{lem}

\begin{lem}
If $A$ and $B$ are densely defined and closed such that either $A$
is invertible or $B$ is bounded, then $AB$ is closed.
\end{lem}

For other related results, see \cite{DevNussbaum-von-Neumann},
\cite{DevNussbaum}, \cite{Gustafson-Mortad-2014-Bull-SM},
\cite{Nussbaum-TAMS-commu-unbounded-normal-1969} and
\cite{Schmudgen-strong-comm-SA-Unbd-PAMS}.

\section{Main Results}

\begin{thm}\label{Main theorem 1}
Let $A$ and $B$ be two self-adjoint operators where only $B$ is
bounded. Assume further that $A$ is positive and that $BA$ is
normal. Then both $BA$ and $AB$ are self-adjoint. Besides one has
$AB=BA$.
\end{thm}

\begin{proof}
We may write
\[A(BA)=(AB)A=(BA)^*A.\]
Since $BA$ is normal, $(BA)^*$ is normal too. Since $D(BA)=D(A)$,
Theorem \ref{Fuglede-Mortad-2003} applies and yields
\[A(BA)^*\subset BAA\]
\text{ or }
\begin{equation}\label{inequation 1}
A^2B\subset BA^2.
\end{equation}

Let us transform the previous into a commutativity between $B$ and
$A^2$ (i.e. $BA^2\subset A^2B$).

Since $BA$ and $(BA)^*$ are normal, Corollary \ref{Fuglede Putnam
classic TN=MT} allows us to write
\[B(BA)^*=B(AB)=(BA)B\Longrightarrow B(BA)=(BA)^*B\]
or
\begin{equation}\label{equation 2}
B^2A=AB^2.
\end{equation}

This tells us that both $B^2A$ and $AB^2$ are self-adjoint.
Continuing we note that

\[B^2A^2=AB^2A=A^2B^2\]

and

\begin{equation}\label{equation B2A4=A4B2}
B^2A^4=B^2A^2A^2=A^2B^2A^2=A^4B^2.
\end{equation}

To prove $B$ commutes with $A^2$, we first show that
$\overline{BA^2}$ is normal. We have

\begin{align*}
(BA^2)^*BA^2= & A^2BBA^2\\
\supset &A^2BA^2B \text{ (by Inclusion (\ref{inequation 1}))}\\
\supset &A^2A^2BB \text{ (by Inclusion (\ref{inequation 1}))}\\
=& A^4B^2\\
=&B^2A^4.
\end{align*}
Passing to adjoints gives

\[(\overline{BA^2})^*\overline{BA^2}=(BA^2)^*\overline{BA^2}\subset
[(BA^2)^*BA^2]^*\subset (B^2A^4)^*=A^4B^2.\]

But $A^4B^2$ is symmetric by Equation (\ref{equation B2A4=A4B2}) (it
is even self-adjoint). Since $\overline{BA^2}$ is closed,
$(\overline{BA^2})^*\overline{BA^2}$ is self-adjoint, and since
self-adjoint operators are maximally symmetric, we immediately
obtain
\begin{equation}
\label{overline BA^2 half normal 1}
(\overline{BA^2})^*\overline{BA^2}=A^4B^2.
\end{equation}

Similarly, we may obtain
\[B^2A^4=A^4B^2=A^2A^2BB\subset A^2BA^2B\subset BA^2A^2B=BA^2(BA^2)^*\]
and passing to adjoints yields

\[\overline{BA^2}(\overline{BA^2})^*=\overline{BA^2}(BA^2)^*\subset
[BA^2(BA^2)^*]^*\subset (B^2A^4)^*=A^4B^2.\]

Similar arguments as above imply that

\begin{equation}\label{overline BA^2 half normal 2}
\overline{BA^2}(\overline{BA^2})^*=A^4B^2.
\end{equation}

By Equations (\ref{overline BA^2 half normal 1}) \& (\ref{overline
BA^2 half normal 2}), we see that $\overline{BA^2}$ is normal and
hence we deduce that $(\overline{BA^2})^*=(BA^2)^*=A^2B$ is normal
too.

Since $A^2B$ is densely defined, we may adjoint Relation
(\ref{inequation 1}) to obtain
\[(A^2B)^*\supset (BA^2)^*=A^2B\]
from which $A^2B$ is symmetric. Besides, it is clearly closed. Since
we have just seen that $A^2B$ is normal, we infer that $A^2B$ is
self-adjoint. Thus, we have arrived at the basic inclusion and
commutativity relation
\[BA^2\subset \overline{BA^2}=(A^2B)^*=A^2B=(BA^2)^*.\]

In particular, we then know from Theorem 10 in \cite{Bernau
JAusMS-1968-square root} (or \cite{Kat}) and the positivity of $A$
that $B$ commutes with $A$, that is,
\[BA\subset AB~(=(BA)^*).\]
But both $BA$ and $(BA)^*$ are normal. Since normal operators are
maximally normal, we obtain $BA=AB$.

Accordingly,
\[BA=AB=(BA)^*=(AB)^*,\]
and this completes the proof.
\end{proof}

Now, we turn to the case of $AB$ normal (keeping all the other
assumptions, except $BA$ normal, as those of Theorem \ref{Main
theorem 1}). The proof is very simple if we impose the very strong
condition of the closedness of $BA$. We have

\begin{cor}
Let $A$ and $B$ be two self-adjoint operators where only $B$ is
bounded. Assume further that $A$ is positive, $AB$ is normal and
that $BA$ is closed. Then both $AB$ and $BA$ are self-adjoint.
Besides one has $AB=BA$.
\end{cor}

\begin{proof}
Since $AB$ is normal, and $B$ is bounded, $(BA)^*$ is clearly
normal. Hence so is
\[(BA)^{**}=\overline{BA}=BA.\]
By Theorem \ref{Main theorem 1}, $BA$ is self-adjoint. Therefore,
\[BA=(BA)^*=AB,\]
that is, $AB$ is self-adjoint.
\end{proof}

One may wonder that there are so many assumptions that $AB$ normal
would certainly imply that $BA$ is closed. This is not the case as
seen just below:

\begin{exa}
Let $A$ be a self-adjoint, positive and boundedly invertible
unbounded operator. Let $B$ be its (bounded) inverse. So $B$ too is
self-adjoint. It is then clear
\[AB=I \text{ and } BA\subset I.\]
Hence $AB$ is self-adjoint (hence normal!) but $BA$ is not closed.
\end{exa}

A natural question is: What if $BA$ is not closed, can we still show
that the normality of $AB$ implies its self-adjointness? As in
Theorem \ref{Main theorem 1}, to show that $AB$ is self-adjoint, it
suffices to show that $BA\subset AB$. One of the ways of obtaining
this is via $BA^2\subset A^2B$ which may be obtained if for instance
we have an intertwining result of the type
\[NA\subset AN^*\Longrightarrow N^*A\subset AN\]
where $N$ is an unbounded normal operator playing the role of $AB$
and $A$ (and also $B$) is self-adjoint. Such an intertwining
relation is, however, not true in general as seen in the next
example (we also note that none of the existing unbounded versions
of the Fuglede-Putnam theorem, as \cite{Mortad-PAMS2003},
\cite{Mortad-Fuglede-Putnam-CAOT-2011},
\cite{Palio-Fuglede-Mortad-2013},
\cite{Paliogiannis-2015-newest-Fuglede} and \cite{STO}, allows us to
get this desired "inclusion").

\begin{exa}(cf. \cite{Mortad-Fuglede-Putnam-CAOT-2011})
Define the following operators $A$ and $N$ by
\[Af(x)=(1+|x|)f(x)\text{ and } Nf(x)=-i(1+|x|)f'(x)\]
(with $i^2=-1$) respectively on the domains
\[D(A)=\{f\in L^2(\R): (1+|x|)f\in L^2(\R)\}\]
and
\[D(N)=\{f\in L^2(\R): (1+|x|)f'\in L^2(\R)\}.\]

Then $A$ is self-adjoint and positive (admitting even an everywhere
defined inverse) and $N$ is normal. We then find that

\[AN^*f(x)=NAf(x)=-i(1+|x|)\text{sgn}(x)f(x)-i(1+|x|)^2f'(x)\]
for any $f$ in the \textit{equal} domains
\[D(AN^*)=D(NA)=\{f\in L^2(\R): (1+|x|)f\in L^2(\R),~(1+|x|)^2f'\in L^2(\R)\}\]
and thus
\[AN^*=NA.\]

However
\[AN\not\subset N^*A \text{ and } AN\not\supset N^*A\]
for
\[ANf(x)=-i(1+|x|)^2f'(x)\]
whereas
\[N^*Af(x)=-2i\text{sgn}(x)(1+|x|)f(x)-i(1+|x|)^2f'(x).\]
\end{exa}

Thus the method of proof of Theorem \ref{Main theorem 1} could not
be applied to the case $AB$ and the approach then had to be
different.  After some investigation of possible counterexamples we
were able instead to establish the affirmative result as follows:

\begin{thm}\label{Main theorem 2}
Let $A$ and $B$ be two self-adjoint operators where only $B$ is
bounded. Assume further that $A$ is positive and that $AB$ is
normal. Then both $\overline{BA}$ and $AB$ are self-adjoint. Besides
one has $AB=\overline{BA}$.
\end{thm}

To prove it, we need a few lemmas which are also interesting in
their own right.

\begin{lem}\label{clos{BA}=(AB)* lemma}
Let $A,B$ be self-adjoint and $B\in B(H)$. If $AB$ is densely
defined, then we have:
\[\overline{BA}=(AB)^*.\]
\end{lem}

\begin{proof}This easily follows from
\[(BA)^*=AB\Longrightarrow (AB)^*=(BA)^{**}=\overline{BA}.\]
\end{proof}

\textit{In all the coming lemmas we assume that $A$ and $B$ are two
self-adjoint operators such that $B\in B(H)$ and that $AB$ is
normal.}

\begin{lem}\label{|B|A subset A|B| lemma}
We have:
\[|B|A\subset A|B|.\]
\end{lem}

\begin{proof}
We may write
\[B(AB)=BAB\subset \overline{BA}B.\]
Since both $AB$ and $\overline{BA}$ are normal, Theorem \ref{Fuglede
Putnam classic} yields
\[B(AB)^*\subset (\overline{BA})^*B=(BA)^*B \text{ or merely } B^2A\subset AB^2.\]
Finally, by \cite{Jablonski et al 2014} (or
\cite{mortad-commutatvity-devinatz-2013}), we obtain
\[|B|A\subset A|B|.\]
\end{proof}

Before giving the next lemmas, let
\[B=U|B|=|B|U\]
be the polar decomposition of the self-adjoint $B$, where $U$ is
unitary (cf. \cite{RUD}). Hence
\[B=U^*|B|=|B|U^*.\]
One of the major points is that $U$ is even self-adjoint. To see
this, just re-do the proof of Theorem 12.35 (b) in \cite{RUD} in the
case of a self-adjoint operator. Then use the (self-adjoint!)
Functional Calculus to get that $U$ is self-adjoint. Another proof
may be found in \cite{Bachir-Segres}.  Therefore, $U=U^*$ and
$U^2=I$.

Let us also agree that any $U$ which appears from now on is the $U$
involved in this polar decomposition of $B$.

\begin{lem}\label{(AB)*=UABU lemma}
We have:
\[(AB)^*=UABU\]
so that
\[(AB)^*U=UAB \text{ and } (AB)U=U(AB)^*.\]
\end{lem}

\begin{proof}
Since $|B|A\subset A|B|$, we have $UBA\subset ABU$. Hence
\[UBAU\subset AB\text{ or } (AB)^*\subset (UBAU)^*.\]
Since $U$ is bounded, self-adjoint and invertible, we clearly have
(by Lemma \ref{(AB)*=B*A* B INVERTIBLE!!})
\[(UBAU)^*=U(BA)^*U=UABU.\]
Since $AB$ is normal, so are $UABU$ and $(AB)^*$ so that
\[(AB)^*\subset UABU\Longrightarrow (AB)^*=UABU\]
because normal operators are maximally normal.
\end{proof}

\begin{lem}\label{|AB|=A|B| lemma}Assume also that $A\geq 0$. Then
$A|B|$ is positive, self-adjoint and we have:
\[|AB|=A|B|.\]
\end{lem}

\begin{proof}
First, remember by Lemma \ref{|B|A subset A|B| lemma} that
$|B|A\subset A|B|$. Hence $A|B|$ is positive and self-adjoint as
both $|B|$ and $A$ are commuting and positive (see e.g. Exercise 23,
Page 113 of \cite{SCHMUDG-book-2012}). Now, by Lemma \ref{(AB)*=UABU
lemma} we have
\[AB(AB)^*=ABUABU=A|B|A|B|=(A|B|)^2.\]
Since $AB$ is normal, we have
\[|AB|^2=(AB)^*AB=(A|B|)^2\]
so that (for instance by Theorem 11 of \cite{Bernau
JAusMS-1968-square root})
\[|AB|=A|B|.\]
\end{proof}

\begin{lem}\label{UAB normal lemma}
The operator $UAB$ is normal.
\end{lem}

\begin{proof}
First, $UAB$ is closed as $U$ is invertible and $AB$ is closed. Now,
\begin{align*}
UAB(UAB)^*&=UAB(AB)^*U\\
&=(AB)^*U(AB)^*U\text{ (by Lemma \ref{(AB)*=UABU lemma})}\\
&=(AB)^*ABU^2\text{ (by Lemma \ref{(AB)*=UABU lemma})}\\
&=(AB)^*(AB).
\end{align*}
On the other hand,
\[(UAB)^*UAB=(AB)^*U^2AB=(AB)^*AB,\]
establishing the normality of $UAB$.
\end{proof}

\begin{lem}\label{U|AB|=|AB|U lemma}
We have:
\[U|AB|=|AB|U.\]
\end{lem}

\begin{proof}
Since $UAB$ is normal, we clearly have
\[UAB(AB)^*U=(AB)^*AB\]
which entails
\[UAB(AB)^*=U(AB)^*AB=(AB)^*ABU,\]
i.e.
\[U|AB|^2=|AB|^2U.\]
Hence (by \cite{Bernau JAusMS-1968-square root}), we are sure at
least that $U|AB|\subset |AB|U$. Since $|AB|$ is self-adjoint, a
similar argument to that used in the proof of Lemma \ref{(AB)*=UABU
lemma} gives us \[U|AB|=|AB|U.\]
\end{proof}

\begin{lem}\label{BA subset AB lemma}Assume also that $A\geq 0$.
Then $B$ commutes with $A$, i.e. $BA\subset AB$.
\end{lem}

\begin{proof}
We have by Lemmas \ref{|AB|=A|B| lemma} \& \ref{U|AB|=|AB|U lemma}
\[U|AB|=|AB|U\Longleftrightarrow UA|B|=A|B|U\Longleftrightarrow UA|B|=AB.\]
Using Lemma \ref{|B|A subset A|B| lemma}
\[U|B|A\subset AB \text{ or } BA\subset AB.\]
\end{proof}

We are now ready to prove Theorem \ref{Main theorem 2}.

\begin{proof}
By Lemma \ref{BA subset AB lemma}, $BA\subset AB$ so that
\[(AB)^*\subset AB.\]
Therefore, $AB$ is self-adjoint as we already know that
$D(AB)=D[(AB)^*]$. Finally, Lemma \ref{clos{BA}=(AB)* lemma} gives
\[AB=\overline{BA}.\]
\end{proof}

The question of the essential self-adjointness of a product of two
self-adjoint operators is not easy. In \cite{Mortad-PAMS2003}, a
three page counterexample was constructed to show that if $A$ and
$B$ are two unbounded self-adjoint operators such that $B\geq 0$,
then the normality of $\overline{AB}$ does not entail its
self-adjointness. Related to the question of essential
self-adjointness of products, the reader may consult
\cite{Moller-ess-s.a.}. Having said this, now we may rephrase the
result of Theorem \ref{Main theorem 2} as follows:

\begin{cor}
Let $A$ and $B$ be two self-adjoint operators where only $B$ is
bounded. Assume further that $A$ is positive and that
$\overline{BA}$ is normal. Then $BA$ is essentially self-adjoint.
\end{cor}

\begin{proof}
Since $\overline{BA}$ is normal, so is $(BA)^*$ or $AB$. Then by
Theorem \ref{Main theorem 2}, $AB$ is self-adjoint. By Lemma
\ref{clos{BA}=(AB)* lemma}, $\overline{BA}=(AB)^*$ so that
$\overline{BA}$ is self-adjoint.
\end{proof}

In the end, we give an answer to an open problem from
\cite{Chellali-Mortad} concerning commutativity up to a factor.

\begin{pro}
Let $A$ and $B$ be self-adjoint operators where $B$ is bounded.
Assume that $BA\subset \lambda AB\neq 0$ where $\lambda\in \C$. Then
$\lambda=1$ if $A$ is positive.
\end{pro}

\begin{proof}
By Proposition 2.2 of \cite{Chellali-Mortad}, we already know that
$AB$ is normal. By Theorem \ref{Main theorem 2}, $AB$ is then
self-adjoint. Now,
\[BA\subset \lambda AB\Longrightarrow \frac{1}{\lambda}BA\subset AB.\]
Hence
\[AB=(AB)^*\subset \frac{1}{\lambda}AB.\]
But $D(AB)=D(\alpha AB)$ for any $\alpha\neq 0$. Therefore,
\[AB=\frac{1}{\lambda}AB \text{ or simply } \lambda=1.\]
\end{proof}

\section{Conclusion}

In this conclusion, we summarize all the related results to the
problem considered in this paper. These are gathered from the
present paper, \cite{Mortad-PAMS2003} and \cite{Mortad-IEOT-2009}:

\begin{thm}Let $A$ and $B$ be two self-adjoint operators. Set $N=AB$
and $M=BA$.
\begin{enumerate}
  \item If $A,B\in B(H)$ (one of them is positive) and $N$ (resp. $M$) is
  normal, then $N$ (resp. $M$) is self-adjoint. In either case, we
  also have $AB=BA$.
  \item If only $B\in B(H)$, $B\geq 0$ and $N$ (resp. $M$) is
  normal, then $N$ (resp. $M$) is self-adjoint. Also $BA\subset AB$
  (resp. $BA=AB$).
  \item If $B\in B(H)$, $A\geq 0$ and $N$ (resp. $M$) is normal,
  then $N$ (resp. $M$) is self-adjoint. Also $BA\subset AB$
  (resp. $BA=AB$).
  \item If $B\in B(H)$ and either $A$ or $B$ is positive, then
  $\overline{M}$ normal gives the essential self-adjointness of $M$.
  \item If both $A$ and $B$ are unbounded and $N$ is normal, then it
  is self-adjoint whenever $B\geq 0$.
  \item If both $A$ and $B$ are unbounded and $\overline{N}$ is normal, then
  $N$ need not be essentially self-adjoint even if $B\geq 0$.
  \item If both $A$ and $B$ are unbounded and $M$ is normal, then it
  not necessarily self-adjoint even when $B\geq 0$.
\end{enumerate}
\end{thm}

\section{Acknowledgements}

Rehder \cite{Reh} showed these results for all bounded self-adjoint
operators and also provided a counterexample to show the necessity
for some positivity. Neither the corresponding author of this paper
(back in \cite{Mortad-PAMS2003}) nor the authors of \cite{Alb-Spain}
were aware of his paper. Therefore Rehder deserves credit for being
the first who investigated this topic in the bounded case (in
particular his use of the Fuglede-Putnam theorem).

In the end, Professor Jan Stochel has recently communicated to us a
new variation (unpublished yet) of our Theorem \ref{Main theorem 2}.
We have also appreciated his useful comments on our results and so
we have done with Professor Konrad Schm\"{u}dgen.

\end{document}